\documentclass[12pt]{amsart}

\usepackage{latexsym, amssymb, amsmath,ulem,soul,esint}

\usepackage{amsfonts, graphicx}
\usepackage{graphicx,color}
\newcommand{\be}{\begin{equation}}
\newcommand{\ee}{\end{equation}}
\newcommand{\beq}{\begin{eqnarray}}
\newcommand{\eeq}{\end{eqnarray}}

\newtheorem{thm}{Theorem}[section]

\newtheorem{lma}{Lemma}[section]

\newtheorem{cor}{Corollary}[section]

\theoremstyle{remark}
\newtheorem{rem}{Remark}[section]
\numberwithin{equation}{section}

\def\be{\begin{equation}}
\def\ee{\end{equation}}
\def\bee{\begin{equation*}}
\def\eee{\end{equation*}}
\def\ol{\overline}
\def\lf{\left}
\def\ri{\right}

\def\by{\mathbf{y}}
\def\bx{\mathbf{x}}

\def\wn{\wt\nabla}

\def\cI{\mathcal{I}}
\def\bn{\mathbf{n}}

\def\cK{\mathcal{K}}

\def\wt{\widetilde}
\def\la{\langle}
\def\ra{\rangle}
\def\p{\partial}

\def\ol{\overline}

\def\e{\varepsilon}
\def\a{{\alpha}}
\def\b{{\beta}}

\def\R{\mathbb{R}}

\def\mS{\mathbb{S}}

\def\ve{\varepsilon}

\def\gsch{g_{\mathrm{Sch}}}

\def\Div{\mathrm{Div}}

\begin{document}

\title{Spacelike CMC surfaces near null infinity of the Schwarzschild spacetime}

\author{Luen-Fai Tam}
\address{The Institute of Mathematical Sciences and the Department of Mathematics \\ The Chinese University of Hong Kong \\ Shatin, Hong Kong, China}
\email{lftam@math.cuhk.edu.hk}

 \renewcommand{\subjclassname}{
  \textup{2010} Mathematics Subject Classification}

\thanks{Research partially supported by Hong Kong RGC General Research Fund \#CUHK
14300420}

\date{\today}
\maketitle
\markboth{Luen-Fai Tam }{Spacelike   constant mean curvature surfaces near null infinity}

\begin{abstract}
Motivated by a result of Treibergs, given a smooth function $f(\by)$ on the standard sphere $\mS^2$, $\by\in \mS^2$,  and any positive constant $H_0$, we construct a spacelike  surface with constant mean curvature $H_0$ in the Schwarzschild spacetime, which is the graph of a function $u(\by, r)$ defined   on $r>r_0$ for some $r_0>0$ in the standard coordinates exterior to the blackhole. Moreover, $u$ has the following asymptotic behavior:
 $$
\lf|u(\by, r)-r_*-\lf(f(\by)+r^{-1}\phi(\by)+\frac12 r^{-2}\psi(\by)\ri)\ri|\le Cr^{-3}
$$
for some $C>0$, where $r_*=r+2m\log(\frac r{2m}-1)$.
 Here $\phi, \psi$ are functions on $\mS^2$ given by $\phi= \frac12\lf(H_0^{-2}+|\nabla_{\mS^2} f|^2_{\mS^2}\ri)$ and
   $\psi=\frac12\lf(H_0^{-2} \Delta_{\mS^2} f+\la \nabla_{\mS^2}|\nabla_{\mS^2}f|^2_{\mS^2},\nabla_{\mS^2} f\ra _{\mS^2}\ri)$.
  In particular, the surface intersects the future null infinity with the cut given by the function $f$. In addition, we prove that the function $u-r_*$ is uniformly Lipschitz near the future null infinity.
\end{abstract}

\section{Introduction}\label{s-intro}

In \cite{Treibergs}, Treibergs proved the following:
Given a $C^2$ function $f(\by)$ on the standard sphere $\mS^{n-1}$ and a constant $H_0>0$ there exists an entire spacelike surface in the Minkowski space $\R^{n,1}$ with constant mean curvature $H_0$ which is the entire graph of a function $u$ such that
\bee
\lim_{r\to\infty} (u(\by, r)-r)=f(\by).
\eee
Here $(\by,r)\in \mS^{n-1}\times (0,\infty)$ is the spherical coordinates of $\R^n$. The result  implies that the surface will intersect the future null infinity at the cut given by $(\by, f(\by))$.

Motivated by this result,  we want to study what one may obtain for Schwarzschild spacetime.
Recall the standard Schwarzschild metric defined on $r>2m>0$ is:

\be\label{e-Sch-metric}
g_\mathrm{Sch}=-\lf(1-\frac{2m}r\ri)dt^2+\lf(1-\frac{2m}r\ri)^{-1}dr^2+r^2 \sigma,
\ee
$-\infty<t<\infty$, where $r=\sum_{i=1}^3(x^i)^2$ with $(x^1,x^2,x^3)\in \R^3$ and $\sigma$ is the standard metric of the unit sphere $\mathbb{S}^2$.   The future null infinity $\cI^+$ of the Schwarzschild spacetime is of the form $ \mS^2\times\R$ with $\mS^2$ being the standard sphere, see \S\ref{s-prelim} for more details. Given a cut $\mathcal{C}$ in $\cI^+$ represented as $(\by, f(\by))$, $\by\in \mS^2$ and $f$ is a function of $\by$, we want to construct a spacelike   constant   mean curvature (CMC) surface with positive constant mean curvature which intersects $\cI^+$ at this cut.  To state our result, let

\be\label{e-rstar}
r_*=r+2m\log(\frac r{2m}-1).
\ee
We obtain the following:

\begin{thm}\label{t-main}
Let $f$ be a smooth function on $\mS^2$. For any constant $H_0>0$, there exists $u(\by, r)$ defined for $\by\in \mS^2$, $r>r_0$,  for some $r_0>2m$ such that the graph of $u$ in the Schwarzschild spacetime is a spacelike hypersurface of constant mean curvature $H_0>0$ with boundary value at the future null infinity given by $f$. More precisely, $u$ satisfies:
\bee
\lim_{r\to\infty}(u( \mathbf{y},r)-r_*)=f(\mathbf{y}),
\eee
for all $\by\in \mS^2$. In fact, there exists $C>0$ such that
$$
\lf|u(\by, r)-r_*-\lf(f(\by)+r^{-1}\phi(\by)+\frac12 r^{-2}\psi(\by)\ri)\ri|\le Cr^{-3}
$$
for all $\by\in\mS^2$, $r>r_0$, where
\bee
\left\{
  \begin{array}{ll}
   \phi= \frac12\lf(H_0^{-2}+|\wn f|^2_{\mS^2}\ri); \\
   \psi=\frac12\lf(H_0^{-2}\wt\Delta f+\la \wn|\wn f|^2_{\mS^2},\wn f\ra_{\mS^2}\ri).
  \end{array}
\right.
 \eee
Here $\wn$ and $\wt\Delta$ are  the covariant derivative and Laplacian of  the standard $\mS^2$ respectively. The inner product is taken with respect to the standard metric of $\mS^2$.

\end{thm}
We should emphasis that unlike \cite{Treibergs}, we can only construct a surface which is   defined near the future null infinity.

In addition to the results on spacelike CMC surfaces in the Minkowski space by Treibergs \cite{Treibergs}, there is a well-known result by Bartnik \cite{Bartnik1984} which states that there exists a complete spacelike maximal hypersurface asymptotic to the spatial infinity  in an asymptotically flat spacetime satisfying  a uniform {\it interior condition} (see \cite[p.169]{Bartnik1984} for the definition). In \cite{AnderssonIriondo1999}, Andersson and Iriondo proved the existence of a complete spacelike CMC surface with positive constant mean curvature on an {\it asymptotically Schwarzschild spacetime} (see \cite[Definition 2.1]{AnderssonIriondo1999}) which satisfies a uniform {\it future interior condition} (see \cite[Definition 4.1]{AnderssonIriondo1999}). The constructed surface intersects the future null infinity at $(\by, f(\by))$ with $f(\by)$=constant. In \cite{BartnikChruscielOMurchada}, Bartnik, Chru\'sciel and \'O Murchada studied  complete spacelike surfaces which are maximal outside a spatially compact set   on certain asymptotically flat spacetimes. Recently, spacelike graph of a function which is asymptotically zero in the Minkowski spacetime $\R^{n,1}$ with prescribed mean curvature outside a compact set in $\R^n$ has been constructed in \cite{BartoloCaponioPomponio} by Bartolo, Caponio and Pomponio. On the other hand, spacelike CMC surfaces in the Schwarzschild spacetime have been studied by many people. In particular, In \cite{LL,LL2} K-W Lee and Y-I Lee gave a complete description of spacelike spherical symmetric constant mean curvature    surfaces in the Kruskal extension of Schwarzschild spacetime. See also the references therein.

In Theorem \ref{t-main}, the constructed surface is asymptotically to a cut in the null infinity. Some higher order rate of approximation is also obtained. The main idea is to construct a good foliation near the future null infinity as in \cite{AnderssonIriondo1999} with good estimates so that one can obtain estimates of the so-called tilt factor of a spacelike surface, using a result in  Bartnik \cite{Bartnik1984}. We also need to construct suitable barrier. Our   construction is to use  the results by Li, Shi and the author in \cite{LST}. Without further assumptions on $f$ one might not be able to construct a better barrier to obtain  a better approximation. See   Remark \ref{r-beta} for details.

A natural question is on the regularity   of the function $u-r_*$. In \cite{Stumbles}, Stumbles constructed spacelike CMC surfaces in the Minkowski spacetime (or nearby spacetime) so that the surfaces are $C^3$ near and up to the future null infinity, provided the cut is represented by $(\by, f(\by))$ with $f$ being close to a constant. One may not expect a $C^4$ regularity by the results in \cite{LST}.
In our case,  the foliation  mentioned above in our construction is given by a time function. From the construction, the so-called tilt factor (see the definition in \S\ref{s-est}) of the constructed surface with respect to this time function is uniformly bounded. Using this fact,  we have the following:
\begin{cor}\label{c-intro}
The function $Q(\by,s)=r_*-u(\by,r)$ with $r=s^{-1}$ is uniformly Lipschitz on $\mS^2\times(0,s_0)$ for some $s_0>0$.
\end{cor}
 This is a corollary of a more general result. See Theorem \ref{t-Lip} for  details. This theorem might also be applied  to the spacelike CMC surface in \cite[Theorem 4.1]{AnderssonIriondo1999}.

 The organization of the paper is as follows. In \S\ref{s-prelim}, we will recall the structure of future null infinity $\mathcal{I}^+$ in the Schwarzschild spacetime and will construct a suitable foliation near $\mathcal{I}^+$. In \S\ref{s-est}, we will give detailed estimation on the foliation which will be used later. In \S\ref{s-cmc surfaces} we will prove Theorem \ref{t-main}. In \S\ref{s-Lip} we will discuss a general  Lipschitzian regularity property of spacelike surfaces near $\mathcal{I}^+$ and prove Corollary \ref{c-intro}.

\section{Future null infinity and a foliation}\label{s-prelim}

\subsection{Future null infinity}\label{ss-metric}
Let us   recall the future null infinity of the Schwarzschild spacetime. We always assume that $\frac{\p}{\p t}$ is future pointing. Consider the retarded null coordinate

\be\label{e-uv}
     v=t-r_*,
\ee
where $r_*$ is given by \eqref{e-rstar}. Let $s=r^{-1}$, then

\be\label{e-metric-null}
\begin{split}
g=g_{\mathrm{Sch}}=&-(1-2ms)d {   v}^2+2s^{-2}  d   v ds+s^{-2} \sigma\\
=&s^{-2}(-s^2(1-2ms)d    v^2+2d   v  ds+\sigma)\\
=&:s^{-2}\bar g,
\end{split}
\ee
with $ 0<s<\frac1{2m}$, $-\infty<v<\infty$.
Here the unphysical metric $\bar g$ is
the product metric:
\be\label{e-unphysical-1}
\ol g=(\sigma_{AB})\oplus\left(
      \begin{array}{cc}
        0 & 1 \\
        1 &-s^2(1-2ms) \\
      \end{array}
    \right),
\ee
where $(\sigma_{AB})$ is the standard metric for $\mS^2$ in local coordinates $y^1, y^2$. So
$y^1,y^2, s,  v$ are    coordinates of the spacetime. We also write $(y^1,y^2, s,   v)$ as $(y^1,y^2,y^3,y^4)$.
$\ol g$ is
a smooth Lorentz metric defined on $\by\in \mS^2$, $s\in [0,1/2m)$, $v\in \R$.     The future null infinity $\mathcal{I}^+$ is identified with the boundary   $s=0$, which is a null hypersurface. For later reference,
\be\label{e-unphysical-2}
(\ol g^{ab})=(\ol g)^{-1}=(\sigma^{AB})\oplus\left(
      \begin{array}{cc}
         s^2(1-2ms) & 1 \\
        1 &0 \\
      \end{array}
    \right).
\ee
where $(\sigma^{AB})$ is the inverse of $(\sigma_{AB})$. Hence for the physical metric, $g^{ab}=s^2\ol g^{ab}$.

\vskip .2cm

{\it Convention}: In the following $  a, b, c\dots$ run from 1 to 4; $i, j, k,\dots$ run from 1 to 3 and   $A, B, C, \dots$ run from 1 to 2. Einstein summation convention will be used.

\vskip .2cm

\subsection{Foliation}\label{ss-foliation}
Given a smooth function $f(\by)$ on $\mathbb{S}^2$. Consider the cut   $\mathcal{C}$ given by $(\by,   f(\by)), \by\in \mathbb{S}^2$ in $\cI^+$. We want to extend it to a spacelike CMC surface in the Schwarzschild spacetime. As in \cite{AnderssonIriondo1999},  we need to construct a suitable foliation near $\cI^+$ related to    $f$.
For $\tau>0$, let

\be\label{e-foliation-1}
 P(  \by, s, \tau)=f(\by)+s\phi(\tau,\by)+\frac1{2!}s^2\psi(\tau,\by),
\ee
where $\phi=P_s, \psi=P_{ss}$ at $s=0$ are smooth functions in $\tau, \by$, given by

\be\label{e-phi-psi}
\left\{
  \begin{array}{ll}
   \phi=-\frac12\lf(\tau^2+|\wn f|^2_{\mS^2}\ri); \\
   \psi=\frac12\lf(\tau^2\wt\Delta f+\la \wn|\wn f|^2_{\mS^2},\wn f\ra_{\mS^2}\ri).
  \end{array}
\right.
 \ee
  The choice of $\phi, \psi$ is motivated by the result in \cite[Theorem 3.1]{LST}, so that if  $\Sigma_\tau$ is the surface given by $(\by, s)\to (\by, s, -P)$ in the $\by, s, v$ coordinates, then $\Sigma_\tau$ is spacelike near $s=0$ and   its mean curvature $H$ is such that at $s=0$, $H=\tau^{-1}$, and $\p_sH=0$.

Direct computations give:
\be\label{e-dP}
\left\{
  \begin{array}{ll}
    P_\tau=-\tau s +\frac12\tau s^2 \wt\Delta f=-\tau s\lf(1-\frac12 s\wt\Delta f\ri);\\
P_s=\phi+s\psi;\\
P_A=f_A+s\phi_A+\frac12 s^2\psi_A, A=1, 2. \\
  \end{array}
\right.
\ee
Here for a smooth function $\theta$ in $\by, s, \tau$, the partial derivative of $\theta$ with respect to $s$ is denoted by $\theta_s$ etc.

Let $0<\tau_1<\tau_2<\infty$ be fixed. Let
$$
M=\{\by\in \mathbb{S}^2, s\in (0,\frac1{2m}),\tau\in (\tau_1,\tau_2)\}=\mS^2\times(0,s_0)\times(\tau_1,\tau_2).
$$
Consider the map $\Phi$ from $M$ to the Schwarzschild spacetime in $\by, s, v$ coordinates defined by:
\be\label{e-Phi}
\Phi(\by, s, \tau)=(\by, s, v(\by,s, \tau))
\ee
with  $v(\by,s,\tau)=-P(\by, s,\tau)$.
\begin{lma}\label{l-parametriztaion} There is $\frac 1{2m}>s_0>0$ depending only on $\tau_1, \tau_2, f$ such that $\Phi$ is a diffeomorphism onto its image. Hence $\Phi(M)$ is parametrized by $\by, s, \tau$. Moreover,
\be\label{e-dtau}
\frac{\p \tau}{\p v}=-\frac1{P_\tau} ; \frac{\p \tau}{\p s}=-\frac{P_s}{P_\tau} ; \frac{\p \tau}{\p y^A}=-\frac{P_A}{P_\tau}, A=1,2.
\ee
\end{lma}

\begin{proof}
It is easy to see that if $s_0>0$ is small enough, then $P_\tau<0$. From this and some computations, it is easy to see the lemma is true.

\end{proof}
Let $s_0>0$ be as in the lemma, then
\be
\Phi(M)=\{(\by, s, v)|\ P(\by, s, \tau_2)<v<P(\by, s,\tau_1)\}.
\ee
By the lemma, we can see that $\tau$ is a smooth function on $\Phi(M)$.

 Given $\tau\in (\tau_1,\tau_2)$, let
\bee
\Sigma_\tau=\{ v=-P(\by, s, \tau)\}
\eee
which is a level surface of $\tau$. For fixed $\tau$, let $F(\by, s,v)=v+P(\by,s,\tau)$. To simplify notation, define
\be\label{e-L}
\begin{split}
L=:&-\lf(2P_s+s^2(1-2ms)P_s^2+|\wn P|^2\ri)=-\ol g(\ol\nabla F,\ol\nabla F)\\
\end{split}
\ee
where $\ol\nabla$ is the derivative with respect to $\ol g$.
Here and later, we simply write $|\wn P|$ instead of $||\wn P||_{\mS^2}$ if this does not cause confusion.

\begin{lma}\label{l-foliation-1}
    There is $\frac1{2m}>s_0>0$ depending only on  $\tau_1$,   $\tau_2$ and $ f$ such that $\Sigma_\tau$ is spacelike in $(0,s_0)$ for $\tau\in (\tau_1,\tau_2)$. In fact,
$$
\nabla \tau=-P_\tau^{-1} \lf(g^{va} +g^{ia}P_i \ri)\p_{y^a},
$$
 and
$$
g(\nabla\tau,\nabla\tau)=-s^2P_\tau^{-2}L.
$$
  Moreover, for all $\tau\in (\tau_1,\tau_2)$, $\Sigma_\tau$ is a smooth up to $\cI^+$ in the sense that $P$ is smooth up to $s=0$, which intersects $\cI^+$ at the cut $ \mathcal{C}$ given by $\{(\by, f(\by))| \ \ \by\in \mathbb{S}^2\}$.

\end{lma}
\begin{proof} First let  $s_0$ be as in Lemma \ref{l-parametriztaion} so that $P_\tau<0$. Recall that $(y^1,y^2, y^3, y^4)=(y^1,y^2, s, v)$. Denote the coordinate frame by $\p_a$.  For $\tau\in (\tau_1,\tau_2)$, by \eqref{e-dtau},  we have
\bee
\begin{split}
\nabla \tau=&g^{ab}\frac{\p \tau}{\p y^a}\p_b  \\
=&\lf(g^{vb}\frac{\p \tau}{\p v}+g^{ib}\frac{\p\tau}{\p y^i}\ri)\p_b \\
=&-\frac1{P_\tau}\lf(g^{vb} +g^{ib}P_i\ri)\p_b .\\
\end{split}
\eee
On the other hand, direct computation shows
\bee
\begin{split}
\la\nabla\tau,\nabla\tau\ra=&s^2\ol g^{ab}\frac{\p\tau}{\p y^a}\frac{\p\tau}{\p y^b}\\
=&-s^2P_\tau^{-2}L.
\end{split}
\eee
By \eqref{e-dP}, $P_\tau=-(\tau s+O(s))$. By \eqref{e-dP} and \eqref{e-L},
 \be\label{e-L-1}
 \begin{split}
 L=&-\lf(-\tau^2-|\wt \nabla f|^2+|\wt\nabla f|^2+O(s)\ri)\\
 =&\tau^2+O(s).
 \end{split}
 \ee
 It is easy to see that  if $0<s_0<\frac1{2m}$ is small enough, depending only   on $\tau_1$,  $\tau_2$ and $  f$, then $\Sigma_\tau$ is spacelike in $0<s<s_0$.
 The last assertion is obvious.
\end{proof}
Let $s_0$ be as in the lemma. Since $\frac{\p}{\p t}=\p_v$, we have
\be\label{e-tau-t}
\begin{split}
g(\nabla \tau,\frac{\p}{\p t})=&-P_\tau^{-1}g((g^{va}+g^{ia}P_i)\p_a,\p_v)\\
=&-P_\tau^{-1}(g^{va}+g^{ia}P_i)g_{av}\\
=&-P_\tau^{-1}\\
>&0.
\end{split}
\ee
So $\tau$ is a time function on $\Phi(M)$ with $\nabla\tau$ being past directed.

\section{Estimates on the foliation}\label{s-est}
Let $s_0$ be as in Lemma \ref{l-foliation-1} so that $\Sigma_\tau$ is spacelike for $0<s<s_0$. Let $M=   \mS^2\times (0,s_0) \times(\tau_1,\tau_2)$. Then $\Phi$ is   a parametrization of $\Phi(M)$, with $\tau$ being a time function. Recall that, if $\theta$ is a function in $\by, s, \tau$, then the partial derivatives will be denoted by $\theta_A, \theta_s, \theta_\tau$ etc.   On the other hand,  when consider $\tau$ as a function of $(y^1,y^2,y^3,y^4)=(y^1,y^2, s, v)$, the derivative of $\theta$ with respect to $y^a$ will be denoted by $\p_a \theta$.
Hence
\be\label{e-chain-rule}
\p_A\theta=\theta_\tau \tau_A+\theta_A,\  \p_s \theta=\theta_\tau \tau_s+\theta_s, \ \ \p_v\theta=\theta_\tau \tau_v.
\ee

Let $T$ be the unit future pointing timelike normal of $\Sigma_\tau$ so that
\be\label{e-T}
T=-\a \nabla\tau
\ee
where $\a>0$ is the {\it lapse} of $\tau$ given by
\be\label{e-lapse}
\a^2=-(\la \nabla \tau,\nabla \tau\ra)^{-1}=s^{-2}P_\tau^2 L^{-1}.
\ee
For a spacelike hypersurface $\Sigma$ with future directed unit normal $\bn$, the {\it tilt factor} $\nu$ with respect to $T$ is defined as $\nu=-\gsch(T,\bn)$.

We want to apply a result of Bartnik \cite{Bartnik1984} to estimate the tilt factor for spacelike surfaces in $\Phi(M)$. First recall the following setting in the Bartnik's work.  In $\Phi(M)$, introduce the Riemannian metric $\Theta$:
\be\label{e-Riemannian}
\Theta=g_{\mathrm{Sch}}+2\omega\otimes \omega
\ee
where $\omega$ is the dual of the unit normal $T$. For example, for a vector field $V$, $||V||_\Theta^2=\sum_{i=1}^3\la V,w_i\ra^2+\la V,T\ra^2$, where $w_1, w_2, w_3$  form an orthonormal basis of $\Sigma_\tau$ with respect to metric induced by the Schwarzschild metric $g$.
In order to apply \cite[Theorem 3.1(iii)]{Bartnik1984} (see also remarks on \cite[p.162]{Bartnik1984}) to a compact spacelike hypersurface $\Sigma$ with smooth boundary $\p \Sigma$ in $\Phi(M)$ so that $\tau$=constant on $\p\Sigma$, we need to estimate the following quantities:

\be\label{e-quantities}
\a, ||\a^{-1}\nabla \a||_\Theta, ||\cK||_{\Theta}, ||\nabla T||_\Theta, ||\nabla\nabla T||_\Theta, ||\vec H_{\p \Sigma}||_\Theta
\ee
where $\cK$ is the second fundamental form of $\Sigma_\tau$ and $\nabla$ is the connection of $\gsch$ and $\vec H_{\p\Sigma}$ is the mean curvature vector of $ \p \Sigma$. We have used the fact that the $\gsch$ is Ricci flat. Our result will be summarized in Theorem \ref{t-est} below.
We proceed as in \cite{AnderssonIriondo1999}.

Since we may cover $\mathbb{S}^2$ with finitely many coordinate neighborhoods, we may work on a coordinate neighborhood first. Hence let us fix   a coordinate neighborhood $U$ with local coordinates $y^1, y^2$.   The coordinate frame  with respect to this coordinate is given by:
\be\label{e-frame-1}
\left\{
  \begin{array}{ll}
e_A=:\Phi_*(\frac{\p}{\p y^A})=-P_{A} \p_v+\p_{A}, A=1, 2;\\
e_3=:\Phi_*(\frac{\p}{\p s })=-P_s\p_v+ \p_s;\\
  e_4=:\Phi_*(\frac{\p}{\p\tau})=-P_\tau \p_v.
  \end{array}
\right.
\ee
Here $\p_a$ are coordinate frames with respect to $y^1, y^2, y^3=s, y^4=v$.  Note that if $\theta$ is a smooth function in $\by, s, \tau$, then $e_A(\theta)=\theta_A$ etc.
Note also that
$e_1, e_2, e_3$ are tangential to $\Sigma_\tau$, i.e. $\tau$=constant.
It is easy to see:
\be\label{e-frame-2}
\left\{
  \begin{array}{ll}
   \p_v=-\frac1{P_\tau}e_4;\\
\p_s=-\frac{P_s}{P_\tau}e_4+e_3; \\
\p_A=-\frac{P_A}{P_\tau}e_4+e_A, \ \   A=1, 2.
  \end{array}
\right.
\ee

We may assume that $\sigma_{AB}$ is smooth up to the boundary of $U$  and that the eigenvalues of  $(\sigma_{AB})$ is bounded below by some constant $\lambda>0$.
\vskip .2cm

{\it Notation}: In the following $c(s^\ell), c_{ab}(s^{\ell}), \dots$ for integers $\ell$ will denote functions of the form $s^{\ell}\Lambda$ where $\Lambda$ is a smooth function in $\by, s, \tau$ in $\ol U\times[0,s_0]\times [\tau_1,\tau_2]$. They may vary from line to line.  For example, in \eqref{e-frame-2}, we have
$$
\p_v=c(s^{-1})e_4,
$$
if $s_0$ is small enough.

\begin{lma}\label{l-basis} In the above setting, for $\by\in U$, then the following are true:
\begin{enumerate}
\item[(i)] The metric $\ol g$ in the frame $e_a$ is given by
\bee\label{e-metric-1}
\left\{
  \begin{array}{ll}
    \ol g(e_A, e_B)=\sigma_{AB}-s^2(1-2ms)P_AP_B, 1\le A, B\le 2;  \\
\ol g(e_A, e_3)=\ol g(e_3, e_A)=-P_A-s^2(1-2ms)P_AP_s, 1\le A\le 2;\\
\ol  g(e_3, e_3)=-2P_s-s^2(1-2ms)P_s^2\\
\ol g(e_A, e_4)=\ol g(e_4, e_A)=-s^2(1-2ms)P_\tau P_A, 1\le A\le 2; \\
\ol g(e_3, e_4)=\ol g(e_4, e_3)=-s^2(1-2ms)P_\tau P_s;\\
\ol g(e_4, e_4)=-s^2(1-2ms)P_\tau^2.
  \end{array}
\right.
\eee
  \item [(ii)] Let $\{\ve_1, \ve_2, \ve_3\}$ be an orthonormal basis for $\Sigma_\tau$ with respect to $\ol g$ obtained from $e_1, e_2, e_3$ using Gram-Schmidt process with respect to the metric induced by $\ol g$. Then $\ve_i=c_{ik}(s^0) e_k$, $e_i=c^{ik}(s^0)\ve_k$.
  \item[(iii)] If $s_0>0$ is small enough depending only $\tau_1, \tau_2$  and $  f$, then
 $\a=1+c(s)$ and
$$
T=c_i(s)e_i+\a^{-1}e_4.
$$

\end{enumerate}

\end{lma}
\begin{proof}   Using  \eqref{e-frame-1} and \eqref{e-unphysical-1}, direct computations give (i).

In the following, we always assume $s_0>0$ is small depending only on $\tau_1, \tau_2$ and $f$. Let $\ol g_{ab}=\ol g(e_a,e_b)$. It is easy to see that $\ol g_{ab}$ can be extended smoothly up to $s=0$. Moreover, at $s=0$, $P_A=f_A$.  Hence at $s=0$, for any $(\xi^1,\xi^2, \xi^3)\in \R^3$, let $f_A=\sigma_{AB}f^B$, for any $\e>0$ we have:
\bee
\begin{split}
\ol g_{ij}\xi^i\xi^j=&\sigma_{AB}\xi^A\xi^B-2f_A \xi^A\xi^3+(\tau^2+|\wn f|^2 )(\xi^3)^2\\
=&\sigma_{AB}\xi^B\xi^B-2\sigma_{AB}f^B \xi^A\xi^3+(\tau^2+\sigma_{AB}f^Af^B )(\xi^3)^2
\\
\ge& \sigma_{AB}\xi^B\xi^B-\lf(\e \sigma_{AB}\xi^A\xi^B+\e^{-1}\sigma_{AB}f^Af^B(\xi^3)^2\ri)
+(\tau^2+\sigma_{AB}f^Af^B )(\xi^3)^2\\
=&(1-\e)\sigma_{AB}\xi^B\xi^B+(\tau^2+(1-\e^{-1}\sigma_{AB}f^Af^B )(\xi^3)^2\\
\ge &C\lf((\xi^1)^2+(\xi^2)^2+(\xi^3)^2\ri),
\end{split}
\eee
for some $C>0$ depending only on $\lambda, \tau_1, \tau_2$ and $ |\wn f|$,    if we choose $\e<1$, $\e$ close to 1 so that $\tau^2+(1-\e^{-1}\sigma_{AB}f^Af^B )\ge \tau^2/2$. On the other hand,  away from $s=0$, $(\ol g_{ij})$ is smooth and positive definite. Let $\ve_i=c_{ik}e_k$ as in the lemma, one can see that $c_{ik}$ are smooth function of $y^a$.
 On the other hand,
 $$
 \delta_{ij}=c_{ik}c_{jl}\ol g_{kl}.
 $$
 Hence $ \ol g^{ij} =c_{ki}c_{kj}$. In particular, $\ol g^{ii}=\sum_{ik}c_{ik}^2$.
From this one can conclude $c_{ik}=c(s^0)$.  Similarly one can prove that $c^{ik}=c(s^0)$.

(iii) By \eqref{e-lapse}, \eqref{e-dP} and \eqref{e-L-1},
\bee
\begin{split}
\a=&-s^{-1}P_\tau L^{-\frac12}\\
=&1+c(s).
\end{split}
\eee

 By Lemma \ref{l-foliation-1}
\bee
\begin{split}
T=&-\a\nabla\tau\\
=&\a P_\tau^{-1}\lf(g^{vb}\p_v+g^{ib}P_i\ri)\p_b\\
=&\a s^2P_\tau^{-1}\bigg[\sigma^{BA}P_B\p_A+(1+s^2(1-2ms)P_s)\p_s+P_s \p_v\bigg]\\
=&\a s^2P_\tau^{-1}\bigg[ \sigma^{BA}P_B(-\frac{P_A}{P_\tau }e_4+e_A)+(1+s^2(1-2ms)P_s)(-\frac{P_s}{P_\tau }e_4+e_3)- \frac{P_s}{P_\tau}e_4 \bigg]\\
=&\a s^2P_\tau^{-1}\lf(\sigma^{BA}P_B e_A+(1+s^2(1-2ms)P_s) e_3\ri)+\a s^2P_\tau^{-2}L e_4\\
=&c(s)e_i+\a^{-1}e_4,
\end{split}
\eee
by (iii), \eqref{e-dP}, \eqref{e-lapse} and \eqref{e-frame-2}. This completes the proof of the lemma.
\end{proof}

Let
\be\label{e-w}
w_i=s \ve_i, \ i=1, 2, 3.
\ee
 Then $w_i$ form an orthonormal frame for $\Sigma_\tau$ with respect to the metric induced by the Schwarzschild metric $g$.
\begin{lma}\label{l-lapse}
 If $s_0$ is small enough, depending only on $\tau_1, \tau_2$  and $  f$,  then $\a , \a^{-1}, ||\nabla \a||_\Theta$ are uniformly bounded in $U\times(0,s_0)\times (\tau_1,\tau_2)$.
 \end{lma}
\begin{proof} The estimates of $\a, \a^{-1}$ follow immediately from Lemma \ref{l-basis}.
Let us estimate the derivatives of $\a$. By Lemma \ref{l-basis}, \eqref{e-frame-1}
and
\bee
\begin{split}
w_i(\a)=&s\ve_i(\a)\\
=&s c_{ik}(s^0)e_k(1+c(s))\\
=&c_i(s).
\end{split}
\eee
\bee
\begin{split}
T(\a)=&(c_i(s)e_i+\a^{-1}e_4) (\a)\\
=&c(s).
\end{split}
\eee
Hence $||\nabla \a||_\Theta$ is uniformly bounded in $U\times(0,s_0)\times(\tau_1,\tau_2)$.
\end{proof}

Let $\cK$ be the second fundamental form of $\Sigma_\tau$. We want to estimate $||\cK||_\Theta$. Since the metric $\ol g$ is a product metric, it is more easy to compute the second fundamental form with respect to $\ol g$. Let us recall the following fact:

\begin{lma}\label{l-2nd-conformal}
Let $\Sigma$ be a spacelike hypersurface in a spacetime $(M,g)$. Suppose $g=e^{2\lambda}\ol g$.  Let $\bn$ be a unit normal of $M$ with respect to $g$. Let $\ol\bn=e^{\lambda}\bn$, which is a unit normal with respect to $\ol g$. Let $\cK, \ol\cK$ be the second fundamental forms of $\Sigma$ with respect to $g, \bn$ and $\ol g, \ol\bn$ respectively. Then for any tangential vector fields $X,Y$, we have
\bee
\cK(X,Y)=e^{\lambda}\lf(\ol\cK(X,Y)+d\lambda(\ol\bn)\ol g(X,Y)\ri).
\eee
\end{lma}
\begin{proof} Let $\nabla, \ol\nabla$ be the   connections of $g, \ol g$ respectively. Then
 any smooth vector fields $X, Y$, we have
$$
\nabla_XY=\ol\nabla_XY+\Gamma(X,Y),
$$
where $\Gamma$ is given by
$$
g(\Gamma(X,Y), Z)= X(\lambda) g (Y,Z)+Y(\lambda)g(X,Z)-Z(\lambda)g(X,Y).
$$
 Let $X,Y$ be tangent to $\Sigma$. Then
\bee
\begin{split}
\cK(X,Y)=&-g(\nabla_X Y,\bn)\\
=&-g(\ol\nabla_XY,\bn)-X(\lambda) g (Y,\bn)-Y(\lambda)g(X,\bn)+\bn(\lambda)g(X,Y)\\
=&-e^{\lambda}\ol g(\ol\nabla_XY,\ol\bn)+e^{\lambda}\ol \bn(\lambda)\ol g(X,Y)\\
=&e^{\lambda}\lf(\ol\cK(X,Y)+d\lambda(\ol\bn)\ol g(X,Y)\ri).
\end{split}
\eee

\end{proof}

In our case, $\bn=T$, $\lambda=-\log s$. Let $\ol\bn=e^\lambda T=s^{-1}T$. Then by Lemma \ref{l-foliation-1},
$$
d\lambda(\ol\bn)=-s^{-2}T(s)=\a s^{-2}\nabla \tau(s)=-\a P_\tau^{-1}\lf(1+s^2(1-2ms)P_s)\ri).
$$
So the second fundamental forms $\cK$, $\ol \cK$ of $\Sigma_\tau$ with respect to $g, \ol g$ are related by:
\be\label{e-2nd-conformal}
\cK=s^{-1}\lf[\ol\cK-\a P_\tau^{-1}  \lf(1+s^2(1-2ms)P_s\ri) \ol g\ri].
\ee
The following lemma basically is contained in \cite{LST}.
\begin{lma}\label{l-2nd-est} Let $\cK$ be the second fundamental form of $\Sigma_\tau$. Then in $U\times(0,s_0)\times(\tau_1,\tau_2)$
\bee
\cK(w_i,w_j)=\tau^{-1}\delta_{ij}+c_{ij}(s),
\eee
where $w_1, w_2, w_3$ are given by \eqref{e-w} which form an orthonormal basis of $\Sigma_\tau$ with respect to $g$. In particular, $||\cK||_{\Theta}$ is uniformly bounded.
\end{lma}
\begin{proof} Let $e_i, \ve_i$ be as in \eqref{e-frame-1} and Lemma \ref{l-basis}. By  Lemma \ref{l-cK} below, we have
$$
\ol\cK(e_i,e_j)=c_{ij}(s^0).
$$
Hence using Lemma \ref{l-basis}, \eqref{e-dP} and \eqref{e-2nd-conformal}, we have
\bee
\begin{split}
\cK(w_i,w_j)=&s^2\cK(\ve_i,\ve_j)\\
=&s\lf(\ol\cK(\ve_i,\ve_j) -\a P_\tau^{-1}  \lf(1+s^2(1-2ms)P_s\ri) \ol g(\ve_i,\ve_j)\ri)\\
=&s\ol\cK(\ve_i,\ve_j)-\a s P_\tau^{-1}  \lf(1+s^2(1-2ms)P_s\ri) \delta_{ij}\\
=&sc_{ik}(s^0)c_{jl}(s^0) \ol\cK(e_i,e_j)+\tau^{-1}\delta_{ij}+c(s)\\
=&\tau^{-1}\delta_{ij}+c(s).
\end{split}
\eee

\end{proof}

\begin{lma}\label{l-cK} With the notation as in Lemma \ref{l-2nd-est}, we have $\cK(e_i,e_j)=c_{ij}(s^0)$ in $U\times(0,s_0)\times(\tau_1,\tau_2)$
\end{lma}
\begin{proof}
Using Lemma \ref{e-unphysical-1}, direction computations show:
\be\label{e-connection-1}
\left\{
  \begin{array}{ll}
\ol\nabla_{\p_A}\p_B=\wn_{\p_A}\p_B, 1\le A,B\le 2;\\
    \ol\nabla_{\p_A}\p_a=  \ol \nabla _{\p_a}\p_B=0,\ 3\le a\le 4, 1\le A\le 2;\\
\ol\nabla_{\p_3}\p_3=0; \\
\ol\nabla_{\p_4}\p_4
=s^3(1-5ms+6m^2s^2)\p_3+s(1-3ms)\p_4\\
\ol\nabla_{\p_3}\p_4=\ol\nabla_{\p_4}\p_3=-s(1-3ms)\p_3.
  \end{array}
\right.
\ee
On the other hand,
\be\label{e-2nd-1}
\begin{split}
\ol\nabla_{e_i}e_j=&\ol\nabla_{(-P_i\p_4+\p_i)}(-P_j\p_4+\p_j)\\
=&P_i\p_4(P_j)\p_4+P_iP_j\ol\nabla_{\p_4}\p_4-\p_i(P_j)\p_4-P_i\ol\nabla_{\p_4}\p_j
+\ol\nabla_{\p_i}\p_j\\
=&\lf[P_i\p_4(P_j)-ms^2P_iP_j-\p_i(P_j)\ri]\p_4-ms^4(1-2ms)P_iP_j\p_3
-P_i\ol\nabla_{\p_4}\p_j+\ol\nabla_{\p_i}\p_j.
\end{split}
\ee
We want to compute $\ol g(\nabla_{e_i}e_j,\ol\bn)$ where $\ol\bn=s^{-1}T$ is the unit normal of $\Sigma_\tau$ with respect to $\ol g$.
By Lemma \ref{l-basis}, $g(T,e_4)=-\a$. By \eqref{e-frame-2},  and the fact that $\ol g(T,e_i)=0$, we have,
\be\label{e-2nd-2}
\left\{
  \begin{array}{ll}
    \ol g(\ol \bn, \p_i)=sg(T,\p_i)=s\a P_\tau^{-1}P_i,  & \hbox{$1\le i\le 3$;} \\
    \ol g(\ol\bn, \p_4)=s  g(T,\p_4)=s\a P_\tau^{-1}.
  \end{array}
\right.
\ee
Moreover,
\bee
\begin{split}
P_A\p_4(P_B)-\p_A(P_B)=&-P_\tau^{-1} P_Ae_4(P_B)+P_\tau^{-1}P_A e_4(P_B)-e_A(P_B)\\
=&-P_{AB}.
\end{split}
\eee
Similarly, for $1\le A\le 2$,
\bee
\begin{split}
P_3\p_4(P_A)-\p_3(P_A) =-P_{As}; P_3\p_4(P_3)-\p_3(P_3)=-P_{ss}.
\end{split}
\eee
Combining these with \eqref{e-2nd-1}, \eqref{e-2nd-2} and \eqref{e-dP}, the results follow.

\end{proof}

Next we want to estimate  of $||\nabla T||_\Theta$ and $||\nabla^2T||_{\Theta}$. First we have the following:

\begin{lma}\label{l-bracket}

  Let $w_i$ be as in \eqref{e-w}. Denote $T$ by $w_4$. Then
\bee
\left\{
  \begin{array}{ll}
   [w_i,w_j] =c_{ijk}(s^0)w_k,   1\le i, j, k\le 3 ; \\
   \hbox{$ [T,w_i]$}=\sum_{a=1}^4c_{ia}(s^0)w_a , 1\le i\le3.
  \end{array}
\right.
\eee

\end{lma}
\begin{proof} Observe that $e_a$ in \eqref{e-frame-1} are coordinate frames with respect to the coordinates $y^1, y^2,s,  \tau$. Hence $[e_a,e_b]=0$. Now by Lemma \ref{l-basis}
\bee
\begin{split}
[w_i,w_j]=&[s\ve_i,s\ve_j]\\
=&[sc_{ik}(s^0)e_k, sc_{jl}(s^0)e_l]\\
=&sc_{ik}(s^0)e_k(sc_{jl}(s^0))e_l-sc_{jl}(s^0)e_l(sc_{ik}(s^0))e_k\\
=&c_{ijk}(s)e_k\\
=&c_{ijk}(s^0)w_k.
\end{split}
\eee

By Lemma \ref{l-basis} again,   we have
\bee
\begin{split}
[T,w_i]=&[c_k(s)e_k+\a^{-1}e_4,s c_{ij}(s^0)e_j]\\
=&c_k(s)e_k+[\a^{-1}e_4, s c_{ij}(s^0)e_j]\\
=&c_{ia}(s^0)w_a
\end{split}
\eee
where we have used the fact that $e_4(s)=0$ and $e_4=\a(T-c_k(s)e_k)$.
\end{proof}

\begin{lma}\label{l-dT}
$||\nabla T||_{\Theta}$ is uniformly bounded in $U\times(0,s_0)\times(\tau_1,\tau_2)$.
\end{lma}
\begin{proof} Let $w_i$ be as in \eqref{e-w}.  To estimate $||\nabla T||_{\Theta}$ it is sufficient to estimate $||  \nabla_{w_i}T||_{\Theta}$ and $||\nabla_TT||_{\Theta}$.
Now
$$
g(\nabla_{w_i}T, T)=0;\ \
g(\nabla_{w_i}T, w_j)=\mathcal{K}(w_i,w_j).
$$
By Lemma \ref{l-2nd-est}, $||  \nabla_{w_i}T||_{\Theta}$ are uniformly bounded for $1\le i\le 3$.

Next, we want to estimate $||  \nabla_TT||_{\Theta}$.  It is easy to see that
$
g(\nabla_TT,T)=0.
$
 Since $g(T,w_i)=0$, we have
\bee
\begin{split}
g (\nabla_TT,w_i)=&-g(T,\nabla_Tw_i)\\
=&-g(T,[T,w_i])+g(T,\nabla_{w_i}T)\\
=&-g(T,[T,w_i]).
\end{split}
\eee
By Lemma \ref{l-bracket}, we conclude that $||\nabla_TT||_{\Theta}$ is uniformly bounded.  This completes the proof of the lemma.

\end{proof}
For $ \nabla\nabla T$, we have:

\begin{lma}\label{l-ddT}  $||\nabla\nabla T||_{\Theta}$ is uniformly bounded in $U\times(0,s_0)\times(\tau_1,\tau_2)$.
\end{lma}
\begin{proof} It is sufficient to prove that for all $1\le a, b\le 4$,
$||\nabla_{w_a}\nabla_{w_b}T||_{\Theta}$ is uniformly bounded. Here $w_4=T$.

(i) \underline{To estimate $||\nabla_T\nabla_TT||_{\Theta}$}:

$$
g (\nabla_T\nabla_T T, T)= -g(\nabla_TT,\nabla_TT),
$$
which is uniformly bounded by Lemma \ref{l-dT}.  On the other hand,

\bee
\begin{split}
g  (\nabla_T\nabla_T T, w_i\ra=&   T(g(\nabla_TT,w_i))-g( \nabla_TT,\nabla_Tw_i) \\
=&T(g(\nabla_TT,w_i))-g(\nabla_TT,[T,w_i])-g(\nabla_TT,\nabla_{w_i}T).
\end{split}
\eee
By Lemmas \ref{l-dT}, \ref{l-bracket}, the last two terms above are uniformly bounded. By Lemma \ref{l-bracket}
\be\label{e-ddT-1}
\begin{split}
T(g(\nabla_TT,w_i))=&-T(g( T,[T,w_i]))\\
=&T(c_i(s^0))\\
=&c_i(s^0),
\end{split}
\ee
by Lemma \ref{l-basis}(iii). Hence $||\nabla_T\nabla_TT||_\Theta$ is uniformly bounded.

\vskip .2cm

(ii) \underline{To estimate $||\nabla_{w_i}\nabla_TT||_{\Theta}$}:

$$
g(\nabla_{w_i}\nabla_TT,T)=-g(\nabla_TT,\nabla_{w_i}T),
$$
which is uniformly bounded by  Lemma \ref{l-dT}.
Next,
$$
g(\nabla_{w_i}\nabla_TT,w_j) =w_i(g( \nabla_TT,w_j))-g( \nabla_TT,\nabla_{w_i}w_j).
$$
The first term  on the RHS is uniformly bounded similar to \eqref{e-ddT-1}.
Consider the second term, we have
\bee
\begin{split}
g(\nabla_TT,\nabla_{w_i}w_j)=&g (\nabla_TT, w_k)\cdot g(\nabla_{w_i}w_j, w_k).
\end{split}
\eee
Now
\be\label{e-bracket}
g(\nabla_{w_i}w_j, w_k)=\frac12\lf(g( [w_i, w_j],w_k)-g([w_i, w_k], w_j)-g([w_j,w_k], w_i)\ri).
\ee
Hence by Lemma \ref{l-dT} and \ref{l-bracket}, the second term on the RHS is also uniformly bounded. So $||\nabla_{w_i}\nabla_T T||_\Theta$ is uniformly bounded.

\vskip .2cm

(iii) \underline{To estimate
$||\nabla_T\nabla_{w_i}T||_{\Theta}$}:

\bee
\begin{split}
g(  \nabla_T\nabla_{w_i}T,T) =&-g( \nabla_{w_i}T,\nabla_TT)\\
=&-g( \nabla_{w_i}T, w_j)\cdot g( \nabla_TT,w_j)\\
=&-\mathcal{K}(w_i,w_j)g(\nabla_TT,w_j),\\
=&c(s^0)
\end{split}
\eee
which is uniformly bounded by Lemmas \ref{l-2nd-est} and  \ref{l-dT}. Next,

\bee
\begin{split}
g(  \nabla_T\nabla_{w_i}T,w_j)=&T( g(\nabla_{w_i}T,w_j))-g(\nabla_{w_i}T,\nabla_Tw_j)\\
=&T(\cK(w_i,w_j))-g(\nabla_{w_i}T,w_k))\cdot g(\nabla_Tw_i,w_k)\\
=&T(\cK(w_i,w_j))-\cK(w_i,w_k)\lf(g([T,w_i],w_k)-\cK(w_i,w_k)\ri),
\end{split}
\eee
which is uniformly  bounded by Lemmas \ref{l-basis}, \ref{l-2nd-est}, and \ref{l-bracket}.

Hence $||\nabla_T\nabla_{w_i}T||_\Theta$ is uniformly bounded.

\vskip .2cm

(iv) \underline{To estimate $||\nabla_{w_i}\nabla_{w_j}T||_\Theta$}:

$$
g( \nabla_{w_i}\nabla_{w_j}T, T)=-g( \nabla_{w_j}T, \nabla_{w_i}T),
$$
which is uniformly bounded by  Lemma \ref{l-2nd-est}.

\bee
\begin{split}
g( \nabla_{w_i}\nabla_{w_j}T, w_k\ra=&w_i(g(  \nabla_{w_j}T,w_k))-g  (\nabla_{w_j}T,\nabla_{w_i}w_k)\\
=&w_i(\cK (w_j,w_k))-\cK(w_j,w_l)\cdot g( \nabla_{w_i}w_k,w_l).
\end{split}
\eee
As before, one can see that this is uniformly bounded. This completes the proof of the lemma.

\end{proof}

Finally, we want to estimate $||\mathbf{H}_{\tau,s}||_\Theta$, where
$\mathbf{H}_{\tau,s}$  is the surface given by $\tau$=constant, $s$=constant.
\begin{lma}\label{l-meancurvvec} $||  \mathbf{H}_{\tau,s}||_\Theta$ is uniformly bounded for $s\in(0,s_0)$, $\tau\in(\tau_1,\tau_2)$ and $\by\in U$.
\end{lma}
\begin{proof} Let $N\subset \Sigma_\tau$ which is the level set of $s$. Let $e_a$, $\ve_i$,  $w_i$ be as in \eqref{e-frame-1}, Lemma \ref{l-basis}, and \eqref{e-w}. Observe that $e_1, e_2$ form a basis for the tangent space of $N$, and $\ve_1, \ve_2, \ve_3$ form an orthonormal basis for $\Sigma_\tau$ obtained by Gram-Schmidt process on $e_1, e_2, e_3$ with respect to $\ol g$. Hence $w_1, w_2$ form an orthonormal basis for the tangent space of $N$. $w_3, T$ form an orthonormal basis for the normal bundle of $N$.

\bee
\begin{split}
\mathbf{ H}_{\tau,s}=&\lf(\sum_{A=1}^2  \nabla_{w_A}w_A\ri)^\perp\\
=&- \sum_{A=1}^2  g(\nabla_{w_A}w_A,T) T+  \sum_{A=1}^2  g(\nabla_{w_A}w_A, w_3) w_3\\
=&\sum_{A=1}^2\cK(w_A,w_A)T+\sum_{A=1}^2  g(\nabla_{w_A}w_A, w_3) w_3.
\end{split}
\eee
By Lemmas \ref{l-2nd-est}, \ref{l-bracket} and \eqref{e-bracket}, we conclude that the lemma is true.

\end{proof}

Since $\mS^2$ can be covered by finitely many coordinate neighborhoods, by Lemmas \ref{l-foliation-1}, \ref{l-lapse}, \ref{l-cK}, \ref{l-dT}, \ref{l-ddT} and \ref{l-meancurvvec}, we have the following:
\begin{thm}\label{t-est} There is $s_0>0$ depending only on $\tau_1, \tau_2,   f$ such that for any $\tau\in (\tau_1,\tau_2)$ the level set $\Sigma_\tau$ is spacelike. Moreover, if $\a$ is the lapse   of the time function $\tau$, $T$ is the future pointing unit normal of $\Sigma_\tau$  and $\vec H_{\tau,s}$ is the mean curvature vector of the surface $\tau=$constant, $s=$constant, then the following are all uniformly bounded in $\mS^2\times(0,s_0)\times(\tau_1,\tau_2)$:
\bee
\a, \a^{-1}, ||\nabla \a||_\Theta, ||\nabla T||_{\Theta}, ||\nabla\nabla T||_\Theta, ||\vec H_{\tau,s}||_\Theta.
\eee
Moreover, the mean curvature $H$ of $\Sigma_\tau$ is given by
$$
H=\tau^{-1}+c(s).
$$

\end{thm}

\section{Construction of CMC surfaces}\label{s-cmc surfaces}

Using $t, \mathbf{x}=(x^1,x^2,x^3)$ as coordinates for the Schwarzschild metric in the form \eqref{e-Sch-metric} with $r=|\mathbf{x}|=\lf(\sum_{i=1}^3(x^i)^2\ri)^\frac12$. In this coordinates,

\be\label{e-Sch-2}
g_{\mathrm{Sch}}=-h dt^2+g_{ij}(x) dx^idx^j,
\ee
where  $h=1-\frac{2m}{r}=1-2ms$ with $s=r^{-1}$ and
\bee
g_{ij}=\delta_{ij}+(h^{-1}-1)r^{-2}x^ix^j.
\eee
\vskip .2cm

{\it Notation}:
In this section, we use $\gsch$ to denote the Schwarzschild metric and $g_{ij}(x)dx^idx^j$ will be denoted by $g$. The inverse of $(g_{ij})$ is denoted by $g^{ij}$. For a function $v$ of $\bx=(x^1,x^2,x^3)$, $Dv=g^{ij}v_i\frac{\p}{\p x^j}$ where $v_i=\frac{\p v}{\p x^i}$. $D^iv=g^{ij}v_j$. $|Dv|^2= D^i vD_iv$. The Hessian of $v$
with respect to $g$ will be denoted by $v_{ij}$.
\vskip .2cm

We will prove Theorem \ref{t-main} for the case $H_0=1$.  The other case is similar. Let $f(\by)$ be a smooth function on $\mS^2$. Let
$$
M=\mS^2\times(0,s_0)\times (\frac12, 2),
$$
and define the map $\Phi$ as in $\S$\ref{ss-foliation} given by $P(\by, s,\tau)$ in \eqref{e-foliation-1} with $f$ replaced by $-f$. Let $s_0$ be as in Lemma \ref{l-parametriztaion}. Let $\b$ be a constant, define

$$
Q(\by, s;\b)=-f(\by)+\phi(\by)s+\frac12 \psi(\by)s^2+\b s^3=P(\by,s,1)+\b s^3,
$$
where $\phi, \psi$ are defined as in \eqref{e-phi-psi} with $f$ replaced by $-f$ and $\tau=1$.
  For fixed $\b$, if $s_0$ is small enough then $P(\by, s,2)<Q(\by, s;\b)<P(\by, s,\frac12)$. Hence the  surface $\Sigma$  given by  $v=-Q(\by,s;\b)$    will be in $\Phi(M)$, provided $s_0$ is small enough depending only on $f$ and the bound of $\b$.

 We want to compute the mean curvature of $\Sigma$. All mean curvature will be computed with respect to future pointing unit normal.
\begin{lma}\label{l-barrier} There exist $\b_1<0, \b_2>0$ depending only on $f$ such that the surfaces $\Sigma_1$, $\Sigma_2$ which are the graphs of $v=-Q_1, v= -Q_2$ respectively, and are in $\Phi(M)$ provided $s_0>0$ small enough depending only on $f$. Here  $Q_1(\by,s)=:Q(\by,s;\b_1)$, $Q_2(\by,s)=:Q(\by,s;\b_1)$.  Moreover, the mean curvature of $\Sigma_1$  is smaller than 1 and the mean curvature of $\Sigma_2$ is larger than 1.

\end{lma}
\begin{proof} For notation simplicity, we denote $Q_1$ by $Q$ and $\Sigma_1$ by $\Sigma$. We may assume that $\Sigma$ is spacelike. Let $H$ be the mean curvature of $\Sigma$. By the computation in \cite[Lemma 2.2]{LST},

\be\label{e-CMC-null-1}
\begin{split}
-3  H L^\frac32=&sL \lf( s^2(1-2ms)Q_{ss}+\wt \Delta Q\ri)-\frac12 s\lf(L_s+s^2(1-2ms)L_sQ_s+\la\wn L,\wn Q\ra\ri)\\
&-s^2LP_s-3L
\end{split}
\ee
where
\bee
L=- (2Q_s+s^2(1-2ms)Q_s^2+|\wt\nabla Q|^2 ).
\eee
One can see that $H$ is  smooth up to $s=0$, provided $s_0>0$ is small enough depending only on $f$. By the choice of $\phi, \psi$ and \cite[Theorem 3.1]{LST}, at $s=0$ $H=1, H_s=0$.

In below, $c, c_k$ will denote   smooth functions in $\by, s$ up to $s=0$, which is independent of $\b$, it may vary from line to line. It is easy to see that
at $s=0$, $Q_s=c, Q_{ss}=c, Q_{sss}=6\b$. At $s=0$, $L=1$, $L_s=c$, $L_{ss}=-12\b+c.$
Therefore, at $s=0$
\bee
\begin{split}
(-3HL^\frac32)_{ss}=&-3H_{ss}L^\frac32-3H(L^\frac32)_{ss}\\
=&-3H_{ss} -  \frac92 (L^\frac12 L_s)_s \\
=&-3H_{ss} +54\b+c;\\
\end{split}
\eee

\bee
\begin{split}
 \lf[sL \lf( s^2(1-2ms)Q_{ss}+\wt \Delta Q\ri)\ri]_{ss}
=&2\lf[L \lf( s^2(1-2ms)Q_{ss}+\wt \Delta Q\ri)\ri]_{s}\\
=&c;
\end{split}
\eee

\bee
\begin{split}
 -\frac12 &\lf[s\lf(L_s+s^2(1-2ms)L_sQ_s+\la\wn L,\wn Q\ra\ri)\ri]_{ss}\\
=&- \lf(L_s+s^2(1-2ms)L_sQ_s+\la\wn L,\wn Q\ra\ri)_s\\
=&12\b+c;
\end{split}
\eee
and
\bee
\begin{split}
(-s^2LQ_s-3L)_{ss}=&36\b+c.
\end{split}
\eee
Hence we have
\bee
\begin{split}
-3H_{ss} +54\b+c_1=&c_2+12\b+c_3+36\b +c_4.
\end{split}
\eee
Or
\be\label{e-Hss}
 H_{ss}=2\b+c.
\ee
First choose $\b=\b_1<0$ so that $2\b_1+c<0$. Then $\b_1$ depends only on $f$.
\be
H =1+\frac12 ( 2\b_1+c)s^2+O(s^3).
\ee
In particular, $H <1$ for $0<s<s_0$ provided $s_0$ is small enough depending only on $f$. Similarly, one can choose $\b_2>0$ so that $2\b_2+c>0$. This completes the proof of the lemma.
\end{proof}

\begin{rem}\label{r-beta}
The construction in the above lemma does not work for higher order. Namely, suppose
\be\label{e-beta}
Q(\by, s)=\sum_{i=0}^k \frac1{i!}f_i(\by) s^i+\b s^{k+1},
\ee
and suppose we can choose $f_i$ so that the mean curvature $H$ satisfies $H=1$, and $\frac{\p^i H}{\p s^i}=0$  for $1\le i\le k-1 $ at $s=0$. Then at $s=0$
\bee
3 \p_s^k H=(3-k)(k+1)!\b+c
\eee
where $c$ is a function of $\by, s$. Note that $(3-k)(k+1)!\le 0$ if $k\ge 3$ in contrast to \eqref{e-Hss}. Another issue is that    in general one cannot find $f_i$ so that $\p_s^iH=0$ at $s=0$ if $k\ge 4$, see \cite[Theorem 3.1]{LST}.
\end{rem}

Let  $t, \bx$ be as in \eqref{e-Sch-2}.
\bee
\nabla t=-h^{-1}\frac{\p}{\p t}.
\eee
Lapse $\wt\a$ for the time function $t$ is given by:
\bee
\wt\a^{-2} =-\gsch(\nabla t,\nabla t)=h^{-1}.
\eee
So $\wt\a=h^{\frac12}.$
The future pointing unit normal of $t=$constant is:
\bee
\wt T=h^{-\frac12}\frac{\p}{\p t}.
\eee
\begin{lma}\label{l-t-tau}
Let $T$ be the future pointing unit normal of $\tau$=constant. Then
$\gsch(T,\wt T)=-s^{-1}L^{-\frac12} h^{-\frac12},$
where $L$ is given by \eqref{e-L} with $s=r^{-1}$.
\end{lma}
\begin{proof} By \eqref{e-lapse}, the lapse of $\tau$ is $\a=-s^{-1}P_\tau L^{-\frac12}$. By \eqref{e-tau-t},
$$
\gsch(\nabla \tau,\frac{\p}{\p t})= P_\tau^{-1}.
$$
Hence

\bee
\gsch(T,\wt T)=-\a \gsch(\nabla \tau ,  h^{-\frac12}\p_t)= -s^{-1}L^{-\frac12} h^{-\frac12}.
\eee
\end{proof}

Consider a  surface given by the graph of $u(\bx)$, where $\mathbf{x}=(x^1,x^2,x^3)$, namely, it is given by $t=u(\mathbf{x})$. Then  it is the level surface of $F(t,\mathbf{x})=t-u(\bx)=0$. Normal is given by

\bee
\nabla F=-h^{-1}\frac{\p}{\p t} -D^iu\frac{\p }{\p x^i}.
\eee
\bee
\gsch(\nabla F,\nabla F)=g^{ab}F_aF_b=-h^{-1}+|Du|^2.
\eee
Hence the surface is spacelike if and only if
\be\label{e-spacelike}
1-h|Du|^2>0.
\ee
If $u$ is spacelike, the
future pointing unit normal is:
\bee
\wt \bn=\lf(h^{-1}-|Du|^2\ri)^{-\frac12}\nabla F=\lf(h^{-1}-|Du|^2\ri)^{-\frac12}\lf(h^{-1}\frac{\p}{\p t}+D^iu\frac{\p }{\p x^i}\ri).
\eee
The tilt factor with respect to $\wt T$ is given by:
\be\label{e-tilt}
\wt \nu=-\gsch(\wt T, \wt \bn)=h^{-\frac12}\lf(h^{-1}-|Du|^2\ri)^{-\frac12}=(1-h|Du|^2)^{-\frac12}.
\ee

Suppose the surface is spacelike, it is more easy to appeal  to \cite[p.160]{Bartnik1984} to obtain the mean curvature equation of $u$. Namely,   its graph has   mean curvature $H$ if and only if:

\bee
3H=\Div\lf(\frac{ U}{(1-|U|^2)^\frac12}\ri)+3\wt \nu H^o+\wt\nu \gsch(U,\nabla_{\wt T}\wt T)+\frac12\wt \nu^3\wt T(|U|^2).
\eee
Here $\Div$ is the divergence with respect to the metric $(g_{ij})$,  and
$
U=\wt\a Du.
$
$|U|$ is the norm with respect to $g$ so that $|U|^2=\wt\a^{2}|Du|^2=h|Du|^2$.
$H^o$ is the mean curvature of $t$=constant, which is zero. Note that $\wt T(|U|^2)=0$, because  $|U|^2$ does not depend on $t$.
$$
\gsch(U,\nabla_{\wt T}\wt T)=-\frac12D^iu D_i\log h.
$$
Therefore, the graph of $u$ has mean curvature $H$ if and only if
\be\label{e-mean curvature}
\Div\lf(\frac{h^\frac12 Du }{(1-h|Du|^2)^\frac12}\ri)  -\frac12(1-h|Du|^2)^{-\frac12}D^iu D_i\log h=3H.
\ee

Hence the mean curvature equation is of the form:

\be\label{e-meancurvature-1}
A^{ij}u_{ij}+B(x, Du)=3h^{-\frac12}(1-h|Du|^2)^\frac12 H.
\ee
where
\bee
\left\{
  \begin{array}{ll}
   A^{ij}=(1-h|Du|^2)g^{lj}+hD^i uD^ju,\\
B(x,Du)=h^{-\frac12}  g(Du,     D(h^\frac12) ) + \frac12\frac{ |Du|^2g(Du, D h)}{(1-h|Du|^2) }.
  \end{array}
\right.
\eee
Here $u_{ij}$ is the Hessian of $u$ with respect to $g$.
\begin{lma}\label{l-elliptic}
Assume the graph of $u$ is spacelike, then any $\mathbf{a}=(a_1, a_2, a_3)$, we have
 \bee
 |\mathbf{a}|^2\ge A^{lj}a_la_j\ge (1-h|Du|^2)|\mathbf{a}|^2,
 \eee
where $a^j=g^{ij}a_i$ and $|\mathbf{a}|^2=a_ia^i$.
\end{lma}

 \begin{proof}
\bee
\begin{split}
A^{lj}a_la_j=&(1-h|Du|^2)|\mathbf{a}|^2+hg^{ij}g^{kl}u_iu_k a_l a_j\\
=&(1-h|Du|^2)|\mathbf{a}|^2+h(\sum_i u_ia^i )^2.\\
\end{split}
\eee
From this the lemma follows.
\end{proof}

Recall the following basic fact, see \cite[Lemma 3.3]{Bartnik1984}:
\begin{lma}\label{l-T123} In a Lorentzian vector space with inner product $\la\, ,\,\ra$,
Let $T_1, T_2, T_3$ be future-directed unit timelike vectors. Then
$$
1\le-\la T_1,T_2\ra\le 2\la T_1, T_3\ra \la T_2, T_3\ra.
$$
\end{lma}

We are now ready to prove Theorem \ref{t-main}:
\begin{proof} For simplicity, we prove the case that $H_0=1$. The other case is similar. Consider the foliation by $P(\by, s, \tau)$ at the beginning of the section with
  $M=\mS\times(0,s_0)\times (\frac12, 2)$. We assume that $s_0$ is chosen such that
  in the retarded null coordinate $v=t-r_*$, $\Phi(\by, s,\tau)=(\by,s,  -P)$ is a diffeomorphism between $M$ and $\Phi(M)$. Note that in terms standard coordinates as in \eqref{e-Sch-metric},
  \vskip .1cm
\be\label{e-Phi-M}
\begin{split}
 \Phi(M)=  \{(\by,r,t)|\ \ \by\in \mS^2, r>\frac1{s_0}, r_*-P(\by,r^{-1}, 2)>t>r_*-P(\by, r^{-1}, \frac12)\}.
\end{split}
\ee
Choose $\b_1<0, \b_2>0$ as in Lemma \ref{l-barrier} and let
  $w_1(\by, r)=r_*-Q_1(\by, r^{-1})$, $w_2(\by, r)=r_*-Q_2(\by, r^{-1} )$ where $Q_1$, $Q_2$ are as in the lemma. Here $s_0$ is chosen so that the conclusion of the lemma is true and so that the conclusion of Theorem \ref{t-est} is also true for $\tau_1=\frac12, \tau_2=2$. Also, let $w(\by,r)=r_*-P(\by, r^{-1},1)$.

Let $\frac12r_0=\frac1{s_0}$. For any $R>r_0$, consider the spacetime $N_R$ given by $\mS^2\times (\frac12 r_0, 2R)\times \R$ with metric induced by the Schwarzschild metric. One can see that the surface given by $t=w(\by, r)$  is spacelike and acausal in $N_R$, see \cite[Corollary 46 ]{ONeill}. That is: no  two different points on the surface are causally related. By \cite[Theorem 5.1]{Gerhardt1983}, we can find smooth function $u_R$ of $\by, r$, with $r_0\le r\le R$ so that the graph of $u_R$ is spacelike
with constant mean curvature 1, so that $u_R$ has the same boundary value as $w(\by, r)$.
Since $\b_1<0, \b_2>0$, we have $Q_1(\by, s)<P(\by, s,1)<Q_2(\by, s)$. We have $w_1>w>w_2$. Moreover, the mean curvature of the graph of $w_1$ is less than 1, and the mean curvature of $w_2$ is larger than 1. By the form of  \eqref{e-meancurvature-1} and the fact that the graphs of $w_1, w_2, u_R$  are all spacelike up to the boundary,  by \eqref{e-spacelike} and Lemma \ref{l-elliptic},  one can apply the comparison principle \cite[Theorem 10.1]{GilbardTrudinger} to conclude that
\be\label{e-barrier}
w_1(\by, r)\ge u_R(\by, r)\ge w_2(\by, r).
\ee
Hence the   graph  of $u_R$ is in $\Phi(M)$. In the $(\by, s,\tau)$ coordinates, this graph is given by $(\by, s, \tau(\by, s))$ with $\by \in \mS^2$, $\frac1R<s<\frac1{r_0}$ and $$u_R(\by, r)=r_*-P(\by, s, \tau(\by, s)), $$ with $r=s^{-1}$. On the boundary $s=\frac1R, \frac1{r_0}$, we have $\tau(\by, s)=1$.

The next step is to prove that $u_R$ will subconverge to a solution of \eqref{e-meancurvature-1} with $H=1$ as $R\to\infty$. In order to do this, by Lemma \ref{l-elliptic} we need to estimate the tilt factor of the surface with respect to the time function $t$. So let $\bn_R$ be the future pointing unit normal of the surface and let
$\nu_R=-\gsch(\bn_R, T)$ be the tilt factor with respect to the time function $\tau$ where $T$ is given by \eqref{e-T}. By Theorem \ref{t-est} and the fact that the Schwarzschild spacetime is Ricci flat, we can apply the Bartnik's gradient estimate   \cite[Theorem 3.1(iii)]{Bartnik1984} to conclude that
\be\label{e-nu-R-1}
\nu_R\le C_1
\ee
for some constant $C_1$ independent of $R$. Let $\wt T$ be the future pointed unit normal of the surface $t$=constant, by Lemma \ref{l-T123} and \eqref{e-tau-t}, we have
\be\label{e-nu-R}
\wt\nu_R=:-\gsch(\bn_R,\wt T)\le 2\gsch(\bn_R,T)\gsch(\wt T,T)\le C_2s^{-1}=C_2r
\ee
for some constant $C_2$ independent of $R$. Here we use the fact that the time function $\tau$ on restricted on the graph is bounded between $\frac34$ and $\frac 23$. Hence by \eqref{e-tilt}, for any open set $\Omega$ with compact closure in the region $r>r_0$, there is a constant $C_3$ independent of $R$ such that
$$
(1-h|Du_R|^2)^{-\frac 12}\le C_3.
$$
In particular, $|Du_R|\le C_4$ in $\Omega$ for some constant $C_4$ independent of $R$. By \eqref{e-meancurvature-1} with $H=1$, we may apply \cite[Theorem 13.6]{GilbardTrudinger} to obtain a uniform H\"older estimate for $u_i$. Using Schauder estimates, we conclude that there is a subsequence $R_k\to \infty$ such that $u_{R_k}$ converge in $C^\infty_{\mathrm{loc}}$ in $\{r>r_0\}$ to function $u$ so that its graph is spacelike and has constant mean curvature 1. By \eqref{e-barrier}, for any $R>r_0$, we have

\bee
 w_2(\by, r)\le u(\by, r)\le w_1(\by, r).
\eee
where $s=r^{-1}$. Hence
\bee
|u(\by,r)-r_*-P(\by,r^{-1},1)|\le \max\{-\b_1, \b_2\}r^{-3}.
\eee
 This completes the proof of the theorem.
\end{proof}
As a corollary of the proof, in particular by \eqref{e-nu-R-1}, we have:
\begin{cor}\label{c-tilt}
Let $u$ be the solution in Theorem \ref{t-main}. The tilt factor of the graph of $u$ with respect to the time function $\tau$ is uniformly bounded by a constant.
\end{cor}

We should remark that Lemma \ref{l-barrier} also implies the following:
\begin{cor}\label{c-decay}
Let $f$ be a smooth function on $\mS^2$. Suppose $u$ is a function defined on $r>r_0$ such that the graph of $u$ in the Schwarzschild spacetime is spacelike with constant mean curvature $H_0>0$ so that $u(r,\by)-r_*\to f(\by)$ are $r\to\infty$. Suppose there is $C>0$ such that
$$
 u(\by, r_i)-r_*-\lf(f(\by)+r_i^{-1}\phi(\by)+\frac12 r_i^{-2}\psi(\by)\ri)-Cr_i^{-3} \le0.
$$
for some $r_i\to\infty$,
where $\phi, \psi$ are as in Theorem \ref{t-main}. Then we have
 $$
\limsup_{r\to\infty}\lf(u(\by, r)-r_*-\lf(f(\by)+r^{-1}\phi(\by)+\frac12 r^{-2}\psi(\by)\ri)-C'r^{-3}\ri)\le0.
$$
for some $C'>0$. Similar result is true for the lower bound estimate.

\end{cor}
\begin{proof} In the proof of Lemma \ref{l-barrier}, we may choose $\b_1>0$ large enough so that $\b_1+c<0$ in the notation in the proof and so that $\b_1>C$. Then by the maximum principle, it is easy to see that the corollary is true.

\end{proof}

\section{Lipschitzian regularity}\label{s-Lip}

We want to prove that the solution $u$ given by Theorem \ref{t-main} is Lipschitz near infinity in the sense that the function $r_*-u$ is Lipschitz up to $s=0$ in the coordinates $\by\in \mS^2, s$ and $v=t-r_*.$ In fact, more general result can be obtained. Here is the setup.
Let $\by, s, v$ be as in \S\ref{s-prelim} before. Consider the metric given by
\be\label{e-asy-Sch-1}
G=\omega^{-2}\ol G
\ee
where $\ol G=\ol g+p$ and $\omega=s(1+c(s^3))$. Here $\ol g$ is the unphysical metric \eqref{e-unphysical-1} and $p=p_{ab}dy^ady^b$ with $p_{ab}=p_{ab}(s^3)$ in the coordinates $\by=(y^1,y^2), y^3=s, y^4=v$. Here $p_{ab}(s^3)$ means that $p_{ab}=s^3\Lambda_{ab}$ where $\Lambda_{ab}$ is a   smooth function  on $\mS^2\times[0,s_0)\times\R$ for some $s_0>0$. Similar definition for $c(s^3)$. Hence for fixed $v_1<v_2$,  on $\mS^2\times(0,s_0)\times(v_1,v_2)$ we have $\ol G^{ab}=\ol g^{ab}+p^{ab}$, with $p^{ab}=p^{ab}(s^3)$, provided $s_0$ is small enough.

Let $f$ be a smooth function on $\mS^2$. Suppose $P(\by,s,\tau)$, $\tau>0$ is such that
\be\label{e-P}
P(\by, s, \tau)=f(\by)-\frac12\lf(\tau^2+|\wt\nabla f|^2\ri)s+s^2c(\by, s,\tau)
\ee
where $c$ is smooth function on $\mS^2\times[0,s_0)\times (0,\infty)$. As before, one can see that for fixed $0<\tau_1<\tau_2$, $(\by, s, \tau)\to (\by, s, v)$ with $v=-P(\by,s,\tau)$ is a diffeomorphism from $M=:\mS^2\times(0,s_0)\times (\tau_1,\tau_2)$ on to its image $\mathcal{N}$, provided $s_0$ is small enough. Its image is:
\bee
\mathcal{N}=\{(\by, s, v)|\ \by\in \mS^2, s\in (0,s_0), P(\by,s, \tau_1)<v< P(\by, s, \tau_2)\}.
\eee
Moreover, in terms of the metric $G$, $\nabla \tau$ is timelike. Here $\nabla$ is the derivative with respect to $G$. Let $T=-\a\nabla \tau$ as before, where $\a^{-2}=-G(\nabla\tau,\nabla\tau)$. We have the following:
\begin{thm}\label{t-Lip}
Suppose $\Sigma$ is a spacelike surface inside $\mathcal{N}$ for some $0<\tau_1<\tau_2$, which is   given by $v+Q(\by,s)=0$, $(\by,s)\in\mS^2\times(0,s_0)$. Suppose the tilt factor of $\Sigma$ with respect to $T$ is bounded, that is suppose $-G(T,\bn)\le C$ on $\Sigma$ for some $C>0$ where $\bn$ is the future pointing unit normal of $\Sigma$. Then $Q$ is uniformly Lipschitz on $\mS^2\times(0, s_1)$ for some $0<s_1<s_0$.
\end{thm}
By Corollary \ref{c-tilt} and the proof of Theorem \ref{t-main} we have:
\begin{cor}\label{c-Lip}
Let $u$ be the solution in Theorem \ref{t-main}. Let $Q(\by,s)=r_*-u(\by, r)$ with $s=r^{-1}$. Then $Q(\by, s)$ is uniformly Lipschitz in $\mS^2\times(0,s_0)$ for some $s_0>0$.
\end{cor}
\begin{proof} Let $u_R$ be as in the proof of Theorem \ref{t-main}. Since $u_R$ converges to $u$ in $C^\infty_{loc}$, by \eqref{e-nu-R-1}, one can conclude that $Q$ satisfies the conditions in the theorem. Hence the corollary is true.

\end{proof}
\begin{rem}\label{r-AO} It seems likely that Theorem \ref{t-Lip} can also be applied to the spacelike CMC surface constructed by Andersson and Iriondo in \cite[Theorem 4.2]{AnderssonIriondo1999}.

\end{rem}

Before we prove Theorem \ref{t-Lip}, we need to obtain some estimates.
Consider the coordinates $t, x^1, x^2, x^3$ with $t=v+r_*$, $r=s^{-1}$ and $\by, r$ are the spherical coordinates of $\R^3$. In the following, we always assume that $\tau_1<\tau<\tau_2$. Hence we are doing estimates in $M$ or $\mathcal{N}$.
\begin{lma}\label{l-G-metric}
$\frac{\p}{\p t}$ is timelike with respect to $G$ provided $s_0$ is small enough. Moreover, if $G_{ij}=G(\frac{\p}{\p x^i}, \frac{\p}{\p x^j})$ is the induced metric on $t$=constant, and if $\wt \a$ is the lapse and $\b^i$ is the shift vector, then $G_{ij}=\delta_{ij}+O(s)$, $\wt\a=1+O(s)$, $\b_i=O(s)$. Here $\b_i=G_{ij}\b^j$.
\end{lma}
\begin{proof}
$\frac{\p}{\p t}=\p_v$. Hence by the assumption on $p$,
\bee
G(\frac{\p}{\p t},\frac{\p}{\p t})=\omega^{-2}(\ol g_{vv}+p_{vv} )<0
\eee
if  $s_0>0$ is small enough.  Since $t=v+r_*$, $\p_v(t)=1, \p_s t=-s^{-2}h^{-1}, \p_{y^A}t=0$, for $A=1, 2$, where $h=1-2ms$ as before. Let $\p_a=\p_{y^a}$. Here $y^1, y^2$ are local coordinates of $\mS^2$, $y^3=s, y^4=v$.
So
\bee
\nabla t=G^{ab}\p_at \p_b=\lf(G^{vb}-s^{-2}h^{-1}G^{sb} \ri)\p_b.
\eee
\bee
\begin{split}
G(\nabla t,\nabla t)=&G^{ab}\p_at\p_bt\\
=&G^{vv}-2s^{-2}G^{vs}h^{-1}+G^{ss}s^{-4}h^{-2}\\
=&\omega^2(\ol G^{vv}-2s^{-2}\ol G^{vs}h^{-1}+\ol G^{ss}s^{-4}h^{-2})\\
=&\omega^2\lf[(\ol g^{vv}+p^{vv})-2s^{-2}(\ol g^{vs}+p^{vs})h^{-1}+(\ol g^{ss}+p^{ss})s^{-4}h^{-2}\ri]\\
=&-1+O(s).
\end{split}
\eee
Hence
\bee
\wt \a=1+O(s).
\eee

Comparing with $\gsch$, we see that in the coordinates $t, x^i$,
\bee
G=s^2\omega^{-2}\gsch+\omega^{-2}p.
\eee
On the other hand,
\bee
\frac{\p}{\p t}=\p_v,
\eee
and
\bee
\frac{\p }{\p x^i}=\frac{\p r}{\p x^i}\frac{\p}{\p r}+\frac{\p y^A}{\p x^i}\frac{\p}{\p_{y^A}}
= \frac{x^i}r\lf(-h^{-1}\p_v-s^2\p_s\ri)+\frac{\p y^A}{\p x^i} \p_{y^A}.
\eee
 Here we have used:
\bee
\frac{\p}{\p r}=\frac{\p v}{\p r}\p_v+\frac{\p s}{\p r}\p_s=-h^{-1}\p_v-s^{2}\p_s.
\eee
Hence
\bee
\begin{split}
\b_i=&G(\frac{\p}{\p t}, \frac{\p }{\p x^i})\\
=&\omega^{-2}p(\frac{\p}{\p t}, \frac{\p }{\p x^i})\\
=&\omega^{-2}p(\p_v,\frac{x^i}r\lf(-h^{-1}\p_v-s^2\p_s\ri)\p_s+\frac{\p y^A}{\p x^i} \p_{y^A})\\
=&\omega^{-2}\lf( \frac{x^i}r h^{-1}p_{vv}- \frac{x^i}r s^2p_{vs}+\frac{\p y^A}{\p x^i}p_{vA}\ri)\\
=&O(s).
\end{split}
\eee
Here we have used the fact that $\frac{\p y^A}{\p x^i}=O(s)$.
\bee
\begin{split}
G_{ij}=& {s^2} \omega^{-2} \gsch(\frac{\p}{\p x^i},\frac{\p}{\p x^j})+\omega^{-2}p(\frac{\p}{\p x^i},\frac{\p}{\p x^j})\\
=&\delta_{ij}+O(s)\\
&+\omega^{-2}p\lf(\frac{x^i}r\lf(-h^{-1}\p_v-s^2\p_s\ri)\p_s+\frac{\p y^A}{\p x^i}\p_{y^A}, \frac{x^j}r\lf(-h^{-1}\p_v-s^2\p_s\ri)\p_s+\frac{\p y^B}{\p x^j}\p_{y^B}\ri)\\
=&\delta_{ij}+O(s).
\end{split}
\eee
This completes the proof of the lemma.

\end{proof}

As before, we consider $\frac{\p}{\p t}$ as future pointing.
\begin{lma}\label{l-TT}
$\nabla\tau$ is future pointing. Let $\a$ be the lapse of the time function $\tau$, we have
$$
\a=1+O(s).
$$ Moreover if  $\wt T=-\wt \a \nabla t$, then
$$
-G(T,\wt T)=(\tau s)^{-1}+O(1).
$$
\end{lma}
\begin{proof} In the coordinates $y^1, y^2, y^3=s, y^4=v$,
\bee
\nabla\tau=G^{ab}\p_{y^a}\tau \p_{y^b}.
\eee
As in Lemma \ref{l-parametriztaion},
\bee
\begin{split}
G(\frac{\p}{\p t},\nabla \tau)=-\frac1{P_\tau}>0,
\end{split}
\eee
provided $s$ small enough.
Hence $\nabla\tau$ is past directed. As in the proof of Lemma \ref{l-parametriztaion}
\bee
\begin{split}
G(\nabla\tau,\nabla\tau)=& G^{ab} \p_{y^a}\tau  \p_{y^b}\tau \\
=&\omega^2(\ol g^{ab}\p_{y^a}\tau  \p_{y^b}\tau+p^{ab}\p_{y^a}\tau  \p_{y^b})\\
=&\omega^2( P_\tau^{-2}(2P_s+s^2h P_s^2+|\wt\nabla P|^2+O(s^{-2}))\\
=& s^2P_\tau^{-2}(2P_s+s^2h P_s^2+|\wt\nabla P|^2)+O(s).
\end{split}
\eee
Hence $\a=1+O(s)$. Next,
\bee
\begin{split}
G(T,\wt T)=&\a\wt\a G(\nabla \tau,\nabla T)\\
=&\a\wt\a \omega^2 (\ol g^{ab}+p^{ab})\p_{y^a}\tau \p_{y^b}t\\
=&\a\wt\a \omega^2(s^{-2}P_\tau^{-1}+ O(1))\\
=&-(\tau s)^{-1}+O(1).
\end{split}
\eee
\end{proof}

\begin{proof}[Proof of Theorem \ref{t-Lip}]  Let $F(\by, s,v)=v+Q(\by, s)$. Then the surface $\Sigma$ given by $F=0$ is spacelike. First, let us  prove that $-G(\nabla F,\nabla F)$ is bounded
on $\Sigma$. For $\tau_1<\tau<\tau_2$, by \eqref{e-dtau}  we have
\be\label{e-dP-2}
\p_v\tau=(\tau s+O(s^2))^{-1}, \p_s\tau=(\tau s+O(s^2))^{-1}P_s, \p_{y^A}\tau=(\tau s+O(s^2))^{-1}P_A,
\ee
$A=1, 2$.
Moreover, $P_s, P_A$ are all bounded. We will work on a coordinate neighborhood $U$ of $\mS^2$, so that the standard metric $\sigma_{AB}$ is bounded from above and the eigenvalues of $(\sigma_{AB})$ is bounded from below by a positive constant on $U$.

Let $T, \wt T$ be as in Lemma \ref{l-TT}. Then by Lemma \ref{l-T123} and the assumption on $-G(T,\bn)$,
\bee
-G(T,\wt T)\le 2G(T,\bn)G(\wt T,\bn)\le -C_1G(\wt T,\bn)
\eee
for some $C_1>0$.
By Lemma \ref{l-TT}, we have
\be\label{e-Tn}
-G(\wt T,\bn)\ge C_2s^{-1}
\ee
for some $C_2>0$. Here and below, we implicitly assume that $0<s<s_0$ with $s_0$ is small enough.

In the $t, x^i$ coordinates, $F=t-r_*+Q=:t-u$. We have
\bee
-G(\nabla F,\nabla F)=\wt\a^{-2}\lf((1+\b_iu^i)^2-\wt\a^2u^iu_i\ri)>0.
\eee
where $u_i=\frac{\p u}{\p x^i}$ and $u^i=G^{ij}u_j$. From this inequality, by Lemma \ref{l-G-metric}, we conclude that $u^iu_i$ is uniformly bounded on $\Sigma$ and hence $\frac32\ge 1+\b_iu^i\ge \frac12>0$ provided $s$ is small enough. We can write
\bee
-G(\nabla F,\nabla F)=\wt\a^{-2}(1+\b^iu_i)^2(1-|U|^2),
\eee
where
$$
U=\frac{\wt\a Du}{1+\b_iu^i}
$$
and $Du=u^i\frac{\p}{\p x^i}$. On the other hand,
\bee
-G(\wt T,\bn)=(1-|U|^2)^{-\frac12}.
\eee
By \eqref{e-Tn} we have
\bee
\begin{split}
1-|U|^2\le C_3 s^2
\end{split}
\eee
for some $C_3>0$.
We conclude that by Lemma \ref{l-G-metric},
\be\label{e-GF}
-G(\nabla F,\nabla F)\le C_4 s^2.
\ee
for some $C_4>0$, because $ 1+\b_iu^i$ is bounded.   By the assumption on the tilt factor with respect to the time function $\tau$, we have
\bee
\begin{split}
C_5\ge& -G(T,\bn)\\
=&-\a (-G(\nabla F,\nabla F))^{-\frac12}G(\nabla\tau,\nabla F)\\
=&-\a (-G(\nabla F,\nabla F))^{-\frac12} \omega^2\lf(\ol G^{ab}\p_{y^a}\tau \p_{y^b}F\ri)\\
=&-\a(-G(\nabla F,\nabla F))^{-\frac12} \omega^2\lf((\ol g^{ab}+p^{ab})\p_{y^a}\tau \p_{y^b}F\ri)\\
=&-\a (-G(\nabla F,\nabla F))^{-\frac12} \omega^2\times\\
&(\ol g^{vv}\p_v \tau\p_v F+\ol g^{vs}(\p_s \tau\p_v F+\p_v\tau\p_sF)+\ol g^{ss}\p_s\tau\p_s F+\ol g^{AB}\p_{y^A}\tau\p_{y^B}F+ q^{b}\p_{y^b}F)\\
=&-\a (-G(\nabla F,\nabla F))^{-\frac12} \omega^2\times\\
& \bigg[\frac1{P_\tau}\lf(-(P_s+ Q_s) -s^2(1-2ms) P_sQ_s- \la\wt\nabla P,\wt\nabla Q\ra\ri) + q^{b}\p_{y^b}F\bigg]
\end{split}
\eee
where $q^{b}=O(s^2) $ and we have used \eqref{e-dP-2}. By \eqref{e-dP-2} and \eqref{e-GF} we conclude that
\bee
-Q_s-s^2(1-2ms)P_sQ_s-\la\wt\nabla P,\wt\nabla Q\ra+q^{b}\p_{y^b}F\le C_6
\eee
for some  constant $C_6>0$. Since $P_s, P_A$ are uniformly  bounded and $\sigma^{AB}Q_AQ_B\ge C\sum_{A=1}^2 Q_A^2$ for some $C>0$, we have for any $\e>0$, we have
\be\label{e-Q-1}
-(1+O(s^2))Q_s-(\e+O(s^2))|\wt\nabla Q|^2\le C_7(\e)
\ee
for some constant $C_7$ which also depends on $\e$.

Since $\Sigma$ is spacelike, we have
$$
G(\nabla F,\nabla F)\le 0.
$$
Computations similar to the above, we have
\bee
2Q_s+s^2(1-2ms)Q_s^2+|\wt\nabla Q|^2+O(s^3)\lf(1+Q_s^2+|\wt\nabla Q|^2\ri) \le 0.
\eee
This implies that $Q_s\le Cs^3$ and
  \be\label{e-Q-2}
  (2+O(s^3))Q_s+(s^2(1-2ms)+O(s^3))Q_s^2+(1+O(s^3))|\wt\nabla Q|^2\le C_8
  \ee
for some $C_8>0$. Multiply \eqref{e-Q-2} by $\delta>0$ and add it to \eqref{e-Q-1}, if $s>0$ is small enough,  we have
\bee
-\lf[1+O(s^2)-\delta(2+O(s^3))\ri]Q_s+\lf[\delta(1+O(s^3))-(\e+O(s^2))\ri]|\wt\nabla Q|^2
\le C_7+\delta C_8
\eee
Let $\delta=2\e$ and $\e=\frac18$, we can conclude that
$$
-Q_s\le C_9
$$
for some $C_9>0$ provided $s$ is small enough. Hence $-C_9\le Q_s\le Cs^3$ if $s$ is small enough. From this and \eqref{e-Q-2}, we conclude that $|\wt\nabla Q|^2$ is uniformly bounded provided $s$ is small enough. This completes the proof of the theorem.
\end{proof}

\end{document}